\RequirePackage{fix-cm}
\documentclass{svjour3} 
\smartqed  % flush right qed marks, e.g. at end of proof
\usepackage{graphicx}
\usepackage{amsmath,amssymb}
\usepackage{todonotes}
\usepackage{enumitem,hyperref}
\usepackage{cite}

\def\tilde{\widetilde}

\def\({\left(}
\def\){\right)}
\def\[{\left[}
\def\]{\right]}

\def\gph{\hbox{}}

\def\tto{\rightrightarrows}

\def\H{H\"{o}lder}

\def\hat{\widehat}

\def\tilde{\widetilde}

\def\R{\mathbb{R}}
\def\N{\mathbb{N}}

\def\gph{\operatorname{gph}}

\def\proj{\operatorname{proj}}

\def\tt{\tilde t}

\renewcommand{\epsilon}{\varepsilon}

\renewcommand{\H}{\mathcal{H}}

\makeatletter
\def\namedlabel#1#2{\begingroup
    #2%
    \def\@currentlabel{#2}%
    \phantomsection\label{#1}\endgroup
}
\makeatother

\begin{document}
\title{Catching-up Algorithm with Approximate Projections for Moreau's Sweeping Processes \thanks{}
}

\titlerunning{Catching-up algorithm with errors}        % if too long for running head

\author{Juan Guillermo Garrido and Emilio Vilches  %        \and  Second Author %etc.
}


\institute{
\textbf{Juan Guillermo Garrido} \at 
           Departamento de Ingenier\'ia Matem\'atica, Universidad de Chile, Santiago, Chile.\email{jgarrido@dim.uchile.cl} \newline              
\textbf{Emilio Vilches} \at
           Instituto de Ciencias de la Ingenier\'ia, Universidad de O'Higgins, Rancagua, Chile.  
              \email{emilio.vilches@uoh.cl}               
}

\date{Received: date / Accepted: date}

\maketitle

\begin{abstract}
In this paper, we develop an enhanced version of the catching-up algorithm for sweeping processes through an appropriate concept of approximate projections. We establish some properties of this notion of approximate projection. Then, under suitable assumptions, we show the convergence of the enhanced catching-up algorithm for prox-regular, subsmooth, and merely closed sets. Finally, we discuss efficient numerical methods for obtaining approximate projections.
Our results recover classical existence results in the literature and provide new insight into the 
 numerical simulation of sweeping processes.

\end{abstract}

\section{Introduction}
Given a Hilbert space $\H$, Moreau's sweeping process is a first-order differential inclusion involving the normal cone to a family of closed moving sets $(C(t))_{t\in [0,T]}$. In its simplest form, it can be written as
\begin{equation}\tag{SP}\label{sp1}
\left\{
\begin{aligned}
\dot{x}(t)&\in -N\left(C(t);x(t)\right) & \textrm{ a.e. } t\in [0,T],\\
x(0)&=x_0\in C(0),
\end{aligned}
\right.
\end{equation}
where $N(C(t);\cdot)$ denotes an appropriate normal cone to the sets $(C(t))_{t\in [0,T]}$.  Since its introduction by J.J. Moreau in \cite{MO1,MO2}, the sweeping process has allowed the development of various applications in contact mechanics, electrical circuits, and crowd motion, among others (see, e.g., \cite{Brogliato-M,Acary-Bon-Bro-2011,Maury-Venel}). Furthermore, so far, we have a well-consolidated existence theory for moving sets in the considerable class of prox-regular sets.  

The most prominent (and constructive) method for finding a solution to the sweeping process is the so-called \emph{catching-up algorithm}.   Developed by J.J. Moreau in \cite{MO2} for convex moving sets, it consists in taking a time discretization $\{t_k^n\}_{k=0}^n$ of the interval $[0,T]$ and defining a piecewise linear and continuous function $x_n\colon [0,T]\to \H$ with nodes
 $$
 x_{k+1}^n := \proj_{C(t_{k+1}^n)}(x_k^n) \textrm{ for all } k\in \{0,\ldots,n-1\}.
 $$
Moreover, under general assumptions, it could be proved that the sequence $(x_n)_n$ converges to the unique solution of \eqref{sp1} (see, e.g., \cite{MR2159846}).

The applicability, from the numerical point of view, of the catching-up algorithm is based on the possibility of calculating an exact formula for the projection to the moving sets.  However, for the majority of sets, the projection onto a closed set is impossible to obtain exactly, and only numerical approximations of it are possible. Since there are still no guarantees on the convergence of the catching-up algorithm with approximate projections, in this paper, we develop a theoretical framework for the numerical approximation of the solutions of the sweeping process using an appropriate concept of approximate projection that is consistent with the numerical methods for the computation of the projection on a closed set.

Regarding numerical approximations of sweeping processes, we are aware of the paper \cite{MR2800713}, where the author proposes an implementable numerical method for the particular case of the intersection of the complement of convex sets, which is used to study crowd motion. Our approach follows a different path and is based on numerical optimization methods to find an approximate projection in the following sense: given a closed set $C\subset \H$,  $\epsilon>0$ and $x\in \H$, we say that $\bar{x}\in C$ is an \emph{approximate projection} of $C$ at $x\in \H$ if 
$$
\Vert x-\bar{x}\Vert^2<\inf_{y\in C}\Vert x-y\Vert^2+\varepsilon.
$$
We observe that the set of approximate projections is always nonempty and can be obtained through numerical methods for optimization.  Hence, in this paper, we study the properties of approximate projections and propose a general numerical method for the sweeping process based on approximate projections. We prove that this algorithm converges in three general cases: (i) prox-regular moving sets (without compactness assumptions), (ii) ball-compact subsmooth moving sets, and (iii) general ball-compact fixed closed sets. Hence, our results cover a wide range of existence results for the sweeping process.   It is worth emphasizing that our method generalizes the catching-up algorithm and provides important insights into the numerical simulation of sweeping processes. 

The paper is organized as follows. Section 2 provides the mathematical tools needed for the presentation of the paper and also develops the theoretical properties of approximate projections. Section 3 is devoted to presenting the proposed algorithm and its main properties. Then, in Section 4, we prove the convergence of the algorithm when the moving set has uniformly prox-regular values (without compactness assumptions). Next, in Section 5, we provide the convergence of the proposed algorithm for ball-compact subsmooth moving sets. Section  6 shows the convergence for a fixed ball-compact set. Finally, in Section 7, we discuss numerical aspects for obtaining approximate projections. The paper ends with concluding remarks. 

\section{Preliminaries}
From now on $\H$ stands for a Hilbert space, whose norm, denoted by $\Vert \cdot \Vert$, is induced by an inner product $\langle\cdot,\cdot\rangle$. The closed (resp. open) ball centered at $x$ with radius $r>0$ is denoted by $\mathbb{B}[x,r]$ (resp. $\mathbb{B}(x,r)$), and the closed unit ball is denoted by $\mathbb{B}$. The notation $\H_w$ stands for $\H$ equipped with the weak topology and $x_n \rightharpoonup x$ denotes the weak convergence of a sequence $(x_n)_n$ to $x$. For a given set $S\subset \H$, the \emph{support} and the \emph{distance function} of $S$ of at $x\in \H$ are defined, respectively, as
$$
\sigma(x,S):= \sup_{z\in S}\langle x,z\rangle \textrm{ and } d_S(x):=\inf_{z\in S}\|x-z\|.
$$
Given $\rho\in ]0,+\infty]$ and $\gamma<1$ positive, the   \emph{$\rho$-enlargement} and the 
 \emph{$\gamma\rho$-enlargement} of $S$ are defined, respectively, as
$$
U_\rho(S) = \{x\in \H:d_S(x)<\rho\} \textrm{ and } U_\rho^\gamma(S):=\{x\in \H:d_S(x)<\gamma \rho\}.
$$
Given $A,B\subset \H$ two sets, we define the \emph{excess} of $A$ over $B$ as the quantity $e(A,B) := \sup_{x\in A} d_B(x)$. From this, we define the \emph{Hausdorff distance} between $A$ and $B$ as
$$d_H(A,B) := \max\{e(A,B),e(B,A)\}.$$
Further properties about Hausdorff distance can be found in \cite[Sec. 3.16]{MR2378491}.\\
\noindent A vector $h\in \H$ belongs to the Clarke tangent cone $T(S;x)$ (see \cite{Clarke1983}); when for every sequence $(x_n)_n$ in $S$ converging to $x$ and every sequence of positive numbers $(t_n)_n$ converging to $0$, there exists a sequence $(h_n)_n$ in $\H$ converging to $h$ such that $x_n+t_nh_n\in S$ for all $n\in \mathbb{N}$. This cone is closed and convex, and its negative polar $N(S;x)$ is the Clarke normal cone to $S$ at $x\in S$, that is,
\begin{equation*}
N\left(S;x\right):=\left\{v\in \H : \left\langle v,h\right\rangle \leq 0 \quad  \textrm{for all } h\in T(S;x)\right\}.
\end{equation*}
As usual, $N(S;x)=\emptyset$ if $x\notin S$. Through that normal cone, the \emph{Clarke subdifferential} of a function $f\colon \H\to \mathbb{R}\cup\{+\infty\}$ is defined by
\begin{equation*}
\partial f(x):=\left\{v\in \H : (v,-1)\in N\left(\operatorname{epi}f,(x,f(x))\right)\right\},
\end{equation*}
where $\operatorname{epi}f:=\left\{(y,r)\in \H\times \mathbb{R} : f(y)\leq r\right\}$ is the epigraph of $f$. When the function $f$ is finite and locally Lipschitzian around $x$, the Clarke subdifferential is characterized (see \cite{Clarke1998}) in the following simple and amenable way
\begin{equation*}
\partial f(x)=\left\{v\in \H : \left\langle v,h\right\rangle \leq f^{\circ}(x;h) \textrm{ for all } h\in \H\right\},
\end{equation*}
where
\begin{equation*}
f^{\circ}(x;h):=\limsup_{(t,y)\to (0^+,x)}t^{-1}\left[f(y+th)-f(y)\right],
\end{equation*}
is the \emph{generalized directional derivative} of the locally Lipschitzian function $f$ at $x$ in the direction $h\in \H$.  The function $f^{\circ}(x;\cdot)$ is in fact the support of $\partial f(x)$. That characterization easily yields that the Clarke subdifferential of any locally Lipschitzian function is a set-valued map with nonempty and convex values satisfying the important property of upper semicontinuity from $\H$ into $\H_w$.\\

\noindent Let $f\colon \H\to \mathbb{R}\cup\{+\infty\}$ be an lsc ({\it lower semicontinuous}) function and $x\in \operatorname{dom}f$. We say that 
\begin{itemize}
\item[(i)] An element $\zeta$ belongs to the \emph{proximal subdifferential}   of $f$ at $x$, denoted by $\partial_P f(x)$, if there exist two nonnegative numbers $\sigma$ and $\eta$ such that
\begin{equation*}
f(y)\geq f(x)+\left\langle \zeta,y-x\right\rangle -\sigma\Vert y-x\Vert^2 \textrm{ for all } y\in \mathbb{B}(x;\eta).
\end{equation*}
\item[(ii)] An element $\zeta\in\H$ belongs to the \emph{Fr\'echet subdifferential} of $f$ at $x$, denoted by $\partial_F f(x)$, if $$\liminf_{h\to 0}\frac{f(x+h)-f(x)-\langle \zeta,h\rangle}{\|h\|}\geq 0.$$ 
\item[(iii)] An element $\zeta$ belongs to the \emph{limiting subdifferential} of $f$ at $x$, denoted by $\partial_L f(x)$,  if there exist sequences $(\zeta_n)$  and $(x_n)$ such that $\zeta_n\in\partial_P f(x_n)$ for all $n\in \mathbb{N}$ and $x_n\to x$,  $\zeta_n \rightharpoonup\zeta$, and $f(x_n)\to f(x)$.

\end{itemize}
Through these concepts, we can define the proximal, Fr\'echet, and limiting normal cone of a given set $S\subset \H$ at $x\in S$, respectively, as
$$
N^P\left(S;x\right):=\partial_P I_S(x),\,  N^F(C;x) := \partial_F I_C(x)  \textrm{ and } N^L(S;x) = \partial_LI_S(x),
$$
where $I_S$ is the indicator function of $S\subset \H$ (recall that $I_S(x)=0$ if $x\in S$ and $I_S(x)=+\infty$ if $x\notin S$). It is well-known that 
\begin{equation}\label{normal-distancia}
N^{P}(S;x)\cap \mathbb{B}=\partial_P d_S(x) \quad \textrm{ for all } x\in S.
\end{equation}

\noindent The equality (see \cite{Clarke1998})
\begin{equation*}%\label{cono-distance}
\begin{aligned}
N\left(S;x\right)&= \overline{\text{co}}^\ast N^L(S;x)=\operatorname{cl}^*\left(\mathbb{R}_+\partial d_S(x)\right) & \textrm{ for } x\in S,
\end{aligned}
\end{equation*}
gives an expression of the Clarke normal cone in terms of the distance function. %As usual, it will be convenient to write $\partial d_S(x)$ in place of $\partial d_S\left(\cdot\right)(x)$.

Now, we recall the concept of uniformly prox-regular sets. Introduced by Federer in \cite{MR110078} and later developed by Rockafellar, Poliquin, and Thibault in \cite{MR1694378}.  The prox-regularity generalizes and unifies convexity and nonconvex bodies with $C^2$ boundary. We refer to \cite{MR2768810,Thibault-2023-II} for a survey.

\begin{definition}
    Let $S$ be a closed subset of $\H$ and $\rho\in ]0,+\infty]$. The set $S$ is called $\rho$-uniformly prox-regular if for all $x\in S$ and $\zeta\in N^P(S;x)$ one has
    \begin{equation*}
        \langle \zeta,x'-x\rangle\leq \frac{\|\zeta\|}{2\rho}\|x'-x\|^2 \textrm{ for all } x'\in S.
    \end{equation*}
\end{definition}
It is important to emphasize that convex sets are $\rho$-uniformly prox-regular for any $\rho>0$. The following proposition provides a characterization of uniformly prox-regular sets (see, e.g.,  \cite{MR2768810}). 
\begin{proposition}\label{prox_reg_prop}
    Let $S\subset \H$ be a closed set and $\rho\in ]0,+\infty]$. The following assertions are equivalent:
    \begin{enumerate}
        \item [(a)] $S$ is $\rho$-uniformly prox-regular.
        \item [(b)] For any positive $\gamma<1$ the mapping $\proj_S$ is well-defined on $U_\rho^\gamma(S)$ and Lipschitz continuous on $U_\rho^\gamma(S)$ with $(1-\gamma)^{-1}$ as a Lipschitz constant, i.e.,
$$\left\|\proj_S\left(u_1\right)-\proj_S\left(u_2\right)\right\| \leq(1-\gamma)^{-1}\left\|u_1-u_2\right\| \quad$$ for all $u_1, u_2 \in U_\rho^\gamma(S)$.
        \item[(c)] For any $x_i \in S, v_i \in N^P\left(S ; x_i\right) \cap \mathbb{B}$, with $i=1,2$,  one has
$$
\left\langle v_1-v_2, x_1-x_2\right\rangle \geq-\frac{1}{\rho}\left\|x_1-x_2\right\|^2,
$$
that is, the set-valued mapping $N^P(S ; \cdot) \cap \mathbb{B}$ is $1 / \rho$-hypomonotone.
\item [(d)] For all $\gamma\in ]0,1[$, for all $x_1,x_2\in U_\rho^\gamma(S)$, for all $\xi\in \partial_P d_S(x)$, one has

    \end{enumerate}
$$\langle\xi,x_2-x_1\rangle\leq \frac{1}{2\rho(1-\gamma)^2}\|x_1-x_2\|^2 + d_S(x_2)-d_S(x_1).$$
\end{proposition}

Next, we recall the class of subsmooth sets that includes the concepts of convex and uniformly prox-regular sets (see  \cite{MR2115366} and also \cite[Chapter~8]{Thibault-2023-II} for a survey).
\begin{definition}
    Let $S$ be a closed subset of $\H$. We say that $S$ is \textit{subsmooth} at $x_0 \in S$, if for every $\varepsilon>0$ there exists $\delta>0$ such that
\begin{equation}\label{eqn_subsmooth}
    \left\langle\xi_2-\xi_1, x_2-x_1\right\rangle \geq-\varepsilon\left\|x_2-x_1\right\|,
\end{equation}
whenever $x_1, x_2 \in \mathbb{B}\left[x_0, \delta\right] \cap S$ and $\xi_i \in N\left(S ; x_i\right) \cap \mathbb{B}$ for $i\in \{1,2\}$. The set $S$ is said \emph{subsmooth} if it is subsmooth at each point of $S$. We further say that $S$ is \textit{uniformly subsmooth}, if for every $\varepsilon>0$ there exists $\delta>0$, such that \eqref{eqn_subsmooth} holds for all $x_1, x_2 \in S$ satisfying $\left\|x_1-x_2\right\|\leq \delta$ and all $\xi_i \in N\left(S ; x_i\right) \cap \mathbb{B}$ for $i\in \{1,2\}$.\\
Let $(S(t))_{t\in I}$ be a family of closed sets of $\H$ indexed by a nonemptyset $I$. The family is called \textit{equi-uniformly subsmooth}, if for all $\epsilon>0$, there exists $\delta>0$ such that for all $t\in I$, the inequality \eqref{eqn_subsmooth} holds for all $x_1,x_2\in S(t)$ satisfying $\|x_1-x_2\|\leq \delta$ and all $\xi_i\in N(S(t);x_i)\cap \mathbb{B}$ with $i\in\{1,2\}$. 
\end{definition}

Given an interval $\mathcal{I}$, a set-valued map $F\colon \mathcal{I}\tto \H$ is said to be measurable if for all open set $U$ of $\H$, the inverse image $F^{-1}(U) = \{t\in\mathcal{I}:F(t)\cap U\neq\emptyset\}$ is a Lebesgue measurable set. When $F$ takes nonempty and closed values and $\H$ is separable, this notion is equivalent to the $\mathcal{L}\otimes \mathcal{B}(\H)$-measurability of the graph  $\gph F := \{(t,x)\in \mathcal{I}\times \H: x\in F(t)\}$  (see, e.g., \cite[Theorem 6.2.20]{MR2527754}). 

Given a set-valued map $F\colon \H\tto \H$, we say $F$ is upper semicontinuous from $\H$ into $\H_w$ if for all weakly closed set $C$ of $\H$, the inverse image $F^{-1}(C)$ is a closed set of $\H$. It is known (see, e.g., see \cite[Proposition 6.1.15 (c)]{MR2527754}) that if $F$ is  upper semicontinuous, then the map $x\mapsto \sigma(\xi,F(x))$ is upper semicontinuous for all $\xi\in \H$. When $F$ takes convex and weakly compact values, these two properties are equivalent (see \cite[Proposition 6.1.17]{MR2527754}).\\
\noindent A set $S\subset \H$ is said ball compact if the set $S\cap r\mathbb{B}$ is compact for all $r>0$. %For a set $S\subset \H$, the distance function of  $S\subset \H$, denoted by $d_S$ is the function defined by $d_{S}(x):=\inf_{y\in S}\Vert x-y\Vert$. 
The \emph{projection} onto $S\subset \H$ is the (possibly empty) set  
$$\operatorname{Proj}_{S}(x):=\left\{z\in S : d_{S}(x)=\Vert x-z\Vert\right\}.$$ 
When the projection set is a singleton, we denote it as $\operatorname{proj}_{S}(x)$. 
 For $\varepsilon>0$, we define the set of \emph{approximate projections}:
\begin{equation*}
\proj_{S}^{\varepsilon}(x):=\left\{z\in S : \Vert x-z\Vert^2 < d_S^2(x)+\varepsilon  \right\}.
\end{equation*}
By definition, the above set is nonempty and open. Moreover, it satisfies similar properties as the projection map (see Proposition \ref{prop-epsilon} below). The approximate projections have been considered several times in variational analysis. In particular, they were used to characterize the subdifferential of the Asplund function of a given set. Indeed, let $S\subset \H$ and consider the Asplund function of the set $S$ 
$$\varphi_S(x) := \frac{1}{2}\|x\|^2-\frac{1}{2}d_S^2(x) \quad x\in \H.
$$ 
Then, the following formula holds (see, e.g., \cite[p.~467]{MR2848527}): 
\begin{equation*}
    \partial \varphi_S(x) = \bigcap_{\epsilon>0}\overline{\text{co}}(\proj_S^\epsilon(x)).
\end{equation*}
We recall that for any set $S\subset \H$ and $x\in \H$, where $\operatorname{Proj}_S(x)\neq \emptyset$,  the following formula holds: 
$$
x-z\in d_S(x)\partial_P d_S(z) \textrm{ for all } z\in \operatorname{Proj}_S(x).
$$
The next result provides an approximate version of the above formula for any closed set $S\subset \H$. 
%The following result will be used in the proof of Theorem \ref{properties_alg} below.
\begin{lemma}\label{proximal-lemma}
 Let $S\subset \H$ be a closed set, $x\in \H$, and $\varepsilon>0$. For each  $z\in \operatorname{proj}_{S}^{\varepsilon}(x)$ there is $v\in  \operatorname{proj}_{S}^{\varepsilon}(x)$ such that  $\Vert z-v\Vert < 2\sqrt{\varepsilon}$ and 
    $$x-z\in (4\sqrt{\epsilon} + d_S(x))\partial_P d_S(v) + 3\sqrt{\epsilon}\mathbb{B}.$$
\end{lemma}
\begin{proof}\quad Fix $\varepsilon>0$, $x\in \H$ and $z\in \proj_{S}^{\varepsilon}(x)$.  According to the Borwein-Preiss Variational Principle \cite[Theorem~2.6]{MR902782} applied to $y\mapsto g(y):=\Vert x-y\Vert^2+ I_S(y)$, there exists $v\in \operatorname{proj}_{S}^{\varepsilon}(x)$ such that $\Vert z-v\Vert<2\sqrt{\varepsilon}$ and $0\in \partial_P g(v) + 2\sqrt{\varepsilon}\mathbb{B}$. Then, by the sum rule for the proximal subdifferential (see, e.g., \cite[Proposition 2.11]{Clarke1998}), we obtain that
$$
x-v\in N^P(S;v)+\sqrt{\varepsilon}\mathbb{B},
$$
which implies that $x-z\in N^P(S;v)+3\sqrt{\varepsilon}\mathbb{B}.
$ 
Next, since $\|x-z\|\leq d_S(x) + \sqrt{\epsilon}$, we obtain that 
$$x-z\in N^P(S;v)\cap (4\sqrt{\epsilon} + d_S(x))\mathbb{B} + 3\sqrt{\epsilon}\mathbb{B}.$$ Finally, the result follows from formula \eqref{normal-distancia} and the above inclusion.
\end{proof}

The following proposition displays some properties of approximation projections for uniformly prox-regular sets.  
\begin{proposition}\label{prop-epsilon}
    Let $S\subset\H$ be a $\rho$-uniformly prox-regular set. Then, one has:
    \begin{enumerate}
        \item [(a)] Let $(x_n)_n$ be a sequence converging to $x\in U_\rho(S)$. Then for any $(z_n)_n$ and any sequence of positive numbers  $(\epsilon_n)_n$ converging to $0$ with $z_n\in \proj_{S}^{\epsilon_n}(x_n)$ for all $n\in\N$, we have that $z_n\to \proj_{S}(x)$.
        \item [(b)] Let $\gamma\in ]0,1[$ and $\epsilon\in ]0,\epsilon_0]$ where $\epsilon_0$ is such that $$\gamma +4\sqrt{\epsilon_0}\left(1+\gamma+\frac{1}{\rho}(1+4\sqrt{\epsilon_0})\right)=1.$$ Then, for all $z_i\in\proj_S^\epsilon(x_i)$ with  $x_i\in U_\rho^\gamma(S)$ for $i\in\{1,2\}$, we have         \begin{equation*}
            (1-\digamma)\|z_1-z_2\|^2\leq \sqrt{\epsilon}\|x_1-x_2\|^2 + M\sqrt{\epsilon}+\langle x_1-x_2,z_1-z_2\rangle,
        \end{equation*}
        where $\digamma := \frac{\alpha}{\rho} + 4\sqrt{\epsilon}\left(1+\frac{\alpha}{\rho}+\frac{1}{\rho}(1+\sqrt{\epsilon})\right)$ with $\alpha := \max\{d_S(x_1),d_S(x_2)\}$ and $M$ is a nonnegative constant only dependent on $\epsilon, \rho,\gamma$.
    \end{enumerate}
\end{proposition}

\begin{proof}
$(a)$ We observe that for all $n\in\N$
\begin{equation*}
    \begin{aligned}
    \|z_n\|&\leq \|z_n-x_n\| + \|x_n\|\leq d_C(x_n) + \sqrt{\epsilon_n} + \|x_n\|.
    \end{aligned}
\end{equation*}
Hence, since $\epsilon_n\to 0$ and $x_n\to x$, we obtain $(y_n)_n$ is bounded. On the other hand, since $x\in U_\rho(S)$, we obtain $\proj_S(x)$ is well-defined and
    \begin{equation*}
        \begin{aligned}
            \|z_n - \proj_S(x)\|^2 &=  \|z_n - x_n\|^2 - \|x_n-\proj_S(x)\|^2 \\
            &+ 2\langle x-\proj_S(x),z_n-\proj_S(x)\rangle+2\langle z_n-\proj_S(x) ,x_n-x\rangle\\
            &\leq   d_S^2(x_n)+\epsilon_n - \|x_n-\proj_S(x)\|^2 \\
            &+ 2\langle x-\proj_S(x),z_n-\proj_S(x)\rangle+2\langle z_n-\proj_S(x) ,x_n-x\rangle\\
            &\leq   \epsilon_n+ 2\langle x-\proj_S(x),z_n-\proj_S(x)\rangle\\
            &+2\langle z_n-\proj_S(x) ,x_n-x\rangle\\
        \end{aligned}
    \end{equation*}
    where we have used $z_n\in \proj_S^{\epsilon_n}(x_n)$ and that $d_S^2(x_n)\leq \Vert x_n-\operatorname{proj}_S(x)\Vert^2$. Moreover, since $x-\proj_S(x)\in N^P(S;\proj_S(x))$ and $S$ is $\rho$-uniformly prox-regular, we obtain that $$2\langle x-\proj_S(x),z_n-\proj_s(x)\rangle\leq\frac{d_S(x)}{\rho}\|z_n-\proj_S(x)\|^2.$$ Therefore, by using the above inequality and rearranging terms, we obtain that
    \begin{equation*}
        \|z_n - \proj_S(x)\|^2\leq \frac{\rho}{\rho-d_S(x)}\left(\epsilon_n  +2\langle z_n-\proj_S(x) ,x_n-x\rangle\right).
    \end{equation*}
  Finally, since $x_n\to x$ and $(z_n)_n$ is bounded, we conclude that $z_n\to \proj_S(x)$.\\
    $(b)$ By virtue of Lemma \ref{proximal-lemma}, for $i\in\{1,2\}$ there exists $v_i,b_i\in \H$ such that 
    $$
    b_i\in\mathbb{B}, v_i\in\proj_S^\epsilon(x_i), \|z_i-v_i\|<2\sqrt{\epsilon} \textrm{ and } \frac{x_i-z_i - 3\sqrt{\epsilon}b_i}{4\sqrt{\epsilon} + d_S(x_i)}\in\partial_P d_S(v_i).$$
    The hypomonotonicity of proximal normal cone (see Proposition \ref{prox_reg_prop} (b)) implies that
    \begin{equation*}
        \left\langle \zeta_1-\zeta_2, v_1-v_2\right\rangle\geq \frac{-1}{\rho}\|v_1-v_2\|^2,
    \end{equation*}
    where $\zeta_i := \frac{x_i-z_i - 3\sqrt{\epsilon}b_i}{4\sqrt{\epsilon}+\alpha}$ for $i\in\{1,2\}$ and $\alpha := \max\{ d_S(x_1),d_S(x_2) \}$. On the one hand, we have
    \begin{equation*}
        \|v_1-v_2\|\leq \|v_1-z_1\| + \|z_1-z_2\| + \|z_2-v_2\|\leq 4\sqrt{\epsilon}+\|z_1-z_2\|,
    \end{equation*}
    and for all $z\in\H$ and $i\in\{1,2\}$
    \begin{equation*}
        |\langle z,v_i-z_i\rangle| \leq \frac{\sqrt{\epsilon}\|z\|^2}{2} + \frac{\|v_i-z_i\|^2}{2\sqrt{\epsilon}}\leq \frac{\sqrt{\epsilon}\|z\|^2}{2} +2\sqrt{\epsilon}.
    \end{equation*}
    On the other hand,
    \begin{equation*}
        \begin{aligned}
            &\langle (x_1-z_1-3\sqrt{\epsilon}b_1) - (x_2-z_2-3\sqrt{\epsilon}b_2),v_1-v_2\rangle\\
            =& \ 3\sqrt{\epsilon}\langle b_2-b_1,v_1-v_2\rangle + \langle (x_1-x_2)-(z_1-z_2),v_1-v_2 \rangle\\
            =& \ 3\sqrt{\epsilon}\langle b_2-b_1,v_1-v_2\rangle + \langle x_1-x_2,v_1-z_1\rangle +\langle x_1-x_2,z_1-z_2\rangle \\
            &+\langle x_1-x_2,z_2-v_2\rangle -\langle z_1-z_2,v_1-z_1\rangle -\|  z_1 - z_2\|^2 - \langle z_1-z_2,z_2-v_2 \rangle\\
            \leq& \ 6\sqrt{\epsilon}(4\sqrt{\epsilon}+\|z_1-z_2\|) + \sqrt{\epsilon}\|x_1-x_2\|^2 + 8\sqrt{\epsilon} + \langle x_1-x_2,z_1-z_2\rangle\\
            &-(1-\sqrt{\epsilon})\|z_1-z_2\|^2\\
            \leq& \ 24\epsilon + 11\sqrt{\epsilon} + \sqrt{\epsilon}\|x_1-x_2\|^2 + \langle x_1-x_2,z_1-z_2\rangle-(1-4\sqrt{\epsilon})\|z_1-z_2\|^2.
        \end{aligned}
    \end{equation*}
    It follows that
    \begin{equation*}
        \begin{aligned}   
        &\left[1-\frac{\alpha}{\rho} -4\sqrt{\epsilon}(1+\frac{1}{\rho}(1+4\sqrt{\epsilon}+\alpha))\right]\|z_1-z_2\|^2\\
        \leq &  \sqrt{\epsilon}\|x_1-x_2\|^2 + \langle x_1-x_2,z_1-z_2 \rangle+4(4\epsilon+\sqrt{\epsilon})(4\frac{\sqrt{\epsilon}}{\rho} + \gamma) +24\epsilon + 11\sqrt{\epsilon}
        \end{aligned}
    \end{equation*}
    which proves the desired inequality.
\end{proof}

\noindent The following result provides a stability result for a family of equi-uniformly subsmooth sets. We refer to see \cite[Lemma 2.7] {MR3574145} for a similar result.
\begin{lemma}\label{upper-semicontinuity}
    Let $\mathcal{C}=\{C_n\}_{n\in \mathbb{N}}\cup \{C\}$ be a family of nonempty, closed and equi-uniformly subsmooth sets. Assume that $$\lim_{n\to\infty}d_{C_n}(x)= 0, \text{ for all } x\in C.$$ Then, for any sequence $\alpha_n \to \alpha\in \R$ and any sequence $(y_n)$ converging to $y$ with $y_n\in C_n$ and $y\in C$, one has
    $$\limsup_{n\to\infty} \sigma(\xi,\alpha_n\partial d_{C_n}(y_n))\leq \sigma(\xi,\alpha\partial d_{C}(y))  \text{ for all } \xi\in \H.$$
\end{lemma}
\begin{proof}
Fix $\xi\in\H$. Since $\partial d_S(x)\subset \mathbb{B}$ for all $x\in \H$, we observe that $$\beta := \limsup_{n\to\infty} \sigma(\xi,\alpha_n\partial d_{C_n}(y_n))<+\infty.$$ Let us consider a subsequence $(n_k)_k$ such that
$$\beta= \lim_{k\to\infty}\sigma(\xi,\alpha_{n_k}\partial d_{C_{n_k}}(y_{n_k})).$$
Given that $\partial d_{C_{n_k}}(y_{n_k})$ is weakly compact for all $k\in\N$, there is $v_{n_k}\in \partial d_{C_{n_k}}(y_{n_k})$ such that $$\sigma(\xi,\alpha_{n_k}\partial d_{C_{n_k}}(y_{n_k})) = \langle \xi,\alpha_{n_k} v_{n_k} \rangle 
 \textrm{ for all } k\in \mathbb{N}.
$$ 
Moreover, the sequence $(v_{n_k})$ is bounded. Hence, without loss of generality, we can assume that $v_{n_k}\rightharpoonup v\in \mathbb{B}$. It follows that $\beta = \langle \xi,\alpha v\rangle$. By equi-uniformly subsmoothness of $\mathcal{C}$, 
 for any $\epsilon>0$, there is $\delta>0$ such that for all $D\in \mathcal{C}$  and $x_1,x_2\in D$ with $\|x_1-x_2\|<\delta$, one has  
 \begin{equation}\label{equi-sub}
 \langle \zeta_1-\zeta_2,x_1-x_2\rangle\geq -\epsilon\|x_1-x_2\|,
 \end{equation}
 whenever $\zeta_i\in N(D;x_i)\cap\mathbb{B}$ for $i\in \{1,2\}$. Next, let $y'\in C$ such that $\|y-y'\|<\delta/2$. Then, since $d_{C_{n_k}}(y')$ converges to $0$, there is a sequence $(y_{n_k}')$ converging to $y'$ with $y_{n_k}'\in C_{n_k}$ for all $k\in\N$.  Hence, there is $k_0\in\N$ such that $\|y_{n_k}'-y'\|<\delta/2$ for all $k\geq k_0$. On the other hand, since $y_n\to y$, then there is $k_0^{\prime}\in\N$ such that 
 $\|y_{n_k}-y\|<\delta/2$ for all $k\geq k_0^{\prime}$. Hence, if $k\geq \max\{k_0,k_0^{\prime}\}=:\hat{k}$ we have $\|y_{n_k}-y_{n_k}'\|<\delta$. Therefore, it follows from the fact that $0\in \partial d_{C_{n_k}}(y_{n_k}')$ and inequality \eqref{equi-sub} that $$\langle v_{n_k},y_{n_k}-y_{n_k}' \rangle\geq -\epsilon\|y_{n_k}-y_{n_k}'\|   \text{ for all }k\geq \hat{k}.$$
By taking $k\to\infty$, we obtain that 
$$\langle v,y-y'\rangle\geq -\epsilon\|y-y'\| \textrm{ for all } y'\in C\cap \mathbb{B}(y,\delta/2),$$
which implies that $v\in N^F(C;y)$. Then, by \cite[Lemma 4.21]{MR2986672}, $$v\in N^F(C;y)\cap\mathbb{B}=\partial_F d_{C}(y)\subset \partial d_{C}(y).$$ 
% \EV{ quien es el $\partial d_C$ sin nombre? Creo que la ultima inclusion es una igualdad.}
% \EV{Aca estoy confundido con el uso de los subdiferenciales. Resulta que cuando los conjuntos son subsmooth los conos de frechet, limiting and clarke coinciden... Esto no est\'a dicho en ninguna parte. Tampoco est\'a la formula que el cono truncado del frechet coincide con el subdiferencial de la distancia... }

\noindent Finally, we have proved that
$$\beta = \langle\xi,\alpha v\rangle\leq \sigma(\xi,\alpha\partial d_{C}(y)),$$
which ends the proof.
 \end{proof}
The following lemma is a convergence theorem for a set-valued map from a topological space into a Hilbert space.

\begin{lemma}\label{lemma_tech_1}
    Let $(E,\tau)$ be a topological space and $\mathcal{G}\colon E\tto \H$ be a set-valued map with nonempty, closed, and convex values. Consider sequences $(x_n)_n\subset E$, $(y_n)_n\subset \H$ and $(\epsilon_n)_n\subset\R_+$ such that 
    \begin{enumerate}
        \item [(i)] $x_n\to x$ (in $E$), $y_n\rightharpoonup y$ (weakly in $\H$) and $\epsilon_n\to 0$;
        \item [(ii)] For all $n\in\N$, $y_n\in \text{co}(\mathcal{G}(x_k)+\epsilon_k\mathbb{B}:k\geq n)$;
        \item [(iii)] $\displaystyle \limsup_{n\to\infty}\sigma(\xi,\mathcal{G}(x_n))\leq \sigma(\xi,\mathcal{G}(x))$ for all $\xi\in\H$.
    \end{enumerate}
    Then, $y\in \mathcal{G}(x)$.
\end{lemma}
\begin{proof}
    Assume by contradiction that $y\notin \mathcal{G}(x)$. By virtue of Hahn-Banach Theorem there exists $\xi\in\H\setminus \{0\}$, $\delta>0$ and $\alpha\in\R$ such that
    \begin{equation*}
        \langle\xi,y'\rangle + \delta\leq \alpha\leq \langle\xi,y\rangle, \ \forall y'\in\mathcal{G}(x). 
    \end{equation*}
    Then, it follows that $\sigma(\xi,\mathcal{G}(x))\leq \alpha-\delta$. Besides, according to (ii) we have for all $n\in\N$, there is a finite set $J_n\subset \N$ such that for all $m\in J_n$, $m\geq n$ and $$y_n =\sum_{j\in J_n}\alpha_j(y_j' + \epsilon_jv_j)$$ where for all $j\in J_n$, $\alpha_j\geq 0$, $v_j\in\mathbb{B}$, $y_j'\in \mathcal{G}(x_j)$ and  $\sum_{j\in J_n}\alpha_j = 1$. Also, there exists $N\in\N$ such that for all $n\geq N$, $\epsilon_n<\frac{\delta}{2\|\xi\|}$. Thus, for $n\geq N$
    \begin{equation*}
        \begin{aligned}
            \langle \xi,y_n\rangle &= \sum_{j\in J_n}\alpha_j\langle\xi,y_j'+\epsilon_jv_j\rangle\\
            & \leq \sum_{j\in J_n}\alpha_j\sup_{k\geq n}\sigma(\xi,\mathcal{G}(x_k)) + \sum_{j\in J_n}\alpha_j\epsilon_j\langle\xi,v_j\rangle\\
            &\leq \sup_{k\geq n}\sigma(\xi,\mathcal{G}(x_k)) + \|\xi\|\sum_{j\in J_n}\alpha_j\frac{\delta}{2\|\xi\|}\leq \sup_{k\geq n}\sigma(\xi,\mathcal{G}(x_k)) + \frac{\delta}{2}.
        \end{aligned}
    \end{equation*}
    Therefore, as $y_n\rightharpoonup y$, letting $n\to\infty$ in the last inequality we obtain that
    \begin{equation*}
        \begin{aligned}
            \langle \xi,y\rangle&\leq \limsup_{n\to\infty}\sigma(\xi,\mathcal{G}(x_n)) + \frac{\delta}{2}\leq \sigma(\xi,\mathcal{G}(x))+\frac{\delta}{2}
        \end{aligned}
    \end{equation*}
    Therefore, $\langle\xi,y\rangle\leq \alpha-\delta/2\leq \langle\xi,y\rangle-\delta/2$, which is a contradiction. The proof is then complete.
\end{proof}
%%%%%%

\section{Catching-up algorithm with errors for sweeping processes}\label{Sec-existence}
In this section, we propose a numerical method for the existence of solutions for the sweeping process:
\begin{equation}\label{Fixed-sweep}
\left\{
\begin{aligned}
\dot{x}(t)&\in -N\left(C(t);x(t)\right)+F(t,x(t)) & \textrm{ a.e. } t\in [0,T],\\
x(0)&=x_0\in C(0),
\end{aligned}
\right.
\end{equation}
where $C\colon [0,T]\tto \H$ is a set-valued map with closed values in a Hilbert space $\H$, $N\left(C(t);x\right)$ stands for  the Clarke normal cone to $C(t)$ at $x$, and  $F\colon [0,T]\times \H \rightrightarrows \H$ is a given set-valued map with nonempty closed and convex values. Our algorithm is based on the catching-up algorithm, except that we do not ask for an exact calculation of the projections.

 The proposed algorithm is given as follows. For $n\in \mathbb{N}^{\ast}$, let $(t_k^n)_{k=0}^n$ be a uniform partition of $[0,T]$ with uniform time step $\mu_n:=T/n$. Let $(\varepsilon_n)_n$ be a sequence of positive numbers such that $\varepsilon_n/\mu_n^2\to 0$.  We consider a sequence of piecewise continuous linear approximations $(x_n)_n$ defined as  $x_n(0)=x_0$ and for any $k\in \{0,\ldots,n-1\}$ and  $t\in ]t_{k}^n,t_{k+1}^n]$
\begin{equation}\label{Prox-algorithm-xn}
    x_n(t)=x_k^n +\frac{t-t_k^n}{\mu_n}\left(x_{k+1}^n-x_k^n-\int_{t_{k}^n}^{t_{k+1}^n}f(s,x_k^n)ds\right) + \int_{t_k^n}^{t}f(s,x_k^n)ds,
\end{equation}
where $x_0^n=x_0$ and
\begin{equation}\label{Prox-algorithm}
x_{k+1}^n\in \proj_{C(t_{k+1}^n)}^{\varepsilon_n}\left(x_k^n +\int_{t_{k}^n}^{t_{k+1}^n}f(s,x_k^n)ds\right)\ \text{for }k\in\{0,1,...,n-1\}.
\end{equation}
Here $f(t,x)$ denotes any selection of $F(t,x)$ such that $f(\cdot,x)$ is measurable for all $x\in\H$. For simplicity, we consider $f(t,x)\in \operatorname{proj}_{F(t,x)}^{\gamma}(0)$ for some $\gamma>0$.  In Proposition \ref{prop_meas}, we prove that it is possible to obtain such measurable selection under mild assumptions.

The above algorithm is called {\it catching-up algorithm with approximate projections} because the projection is not necessarily exactly calculated. We will prove that the above algorithm converges for several families of algorithms as long as the inclusion \eqref{Prox-algorithm} is verified.

 Let us consider functions $\delta_n(\cdot)$ and $\theta_n(\cdot)$ defined as
\begin{equation*}
\delta_n(t)=\begin{cases}
t_k^n & \textrm{ if } t\in [t_k^n,t_{k+1}^n[\\
t_{n-1}^n & \textrm{ if } t=T,
\end{cases}
\text{ and }
\theta_n(t)=\begin{cases}
t_{k+1}^n & \textrm{ if } t\in [t_k^n,t_{k+1}^n[\\
T & \textrm{ if } t=T.
\end{cases}
\end{equation*}

\noindent In what follows, we show useful properties satisfied for the above algorithm, which will help us to prove the existence of the sweeping process (\ref{Fixed-sweep}) in three cases:
\begin{enumerate}
    \item [(i)] The set-valued map $t\tto C(t)$ takes uniformly prox-regular values. 
    \item [(ii)] The set-valued map $t\tto C(t)$ takes subsmooth and ball-compact values.
    \item [(iii)] $C(t) \equiv C$ in $[0,T]$ and $C$ is ball-compact. 
\end{enumerate}

Throughout this section, $F\colon [0,T]\times \H\rightrightarrows \H$ will be a set-valued map with nonempty, closed, and convex values. Moreover, we will consider the following conditions:
\begin{enumerate}
\item[\namedlabel{FixH1}{$(\mathcal{H}_1^F)$}] For all $t\in [0,T]$, $F(t,\cdot)$ is upper semicontinuous from $\H$ into $\H_w$.
 % \item For all $x\in \H$, $F(\cdot,x )$ is a measurable set valued map from $[0,T]$ into $\H$.
\item[\namedlabel{FixH2}{$(\mathcal{H}_2^F)$}] There exists $h\colon \H \to \mathbb{R}^+$ Lipschitz continuous (with constant $L_h>0$) such that
  \begin{equation*}
    \begin{aligned}
    d\left(0,F(t,x)\right):=\inf\{\Vert w\Vert : w\in F(t,x)\}\leq h(x) ,
    \end{aligned}
  \end{equation*}
  for all $x\in \H$ and a.e. $t\in [0,T]$.
  \item [\namedlabel{FixH3}{$(\mathcal{H}_3^F)$}] There is $\gamma>0$ such that the set-valued map $(t,x)\tto \proj_{F(t,x)}^\gamma(0)$ has a selection $f\colon [0,T]\times\H\to \H$ such that $f(\cdot,x)$ is measurable for all $x\in\H$.  
\end{enumerate}
The following proposition provides conditions for the feasibility of hypothesis \ref{FixH3}.
\begin{proposition}\label{prop_meas}
    Let us assume that $\H$ is a separable Hilbert space. Moreover we suppose $F(\cdot,x)$ is measurable for all $x\in\H$, then \ref{FixH3} holds for all $\gamma>0$.
\end{proposition}
\begin{proof}
    Let $\gamma>0$ and fix $x\in \H$. Since the set-valued map $F(\cdot,x)$ is measurable, the map $t\mapsto d(0,F(t,x))$ is a measurable function. Let us define the set-valued map $\mathcal{F}_x\colon t\tto \proj_{F(t,x)}^{\gamma}(0)$. Then, 
    \begin{equation*}
        \begin{aligned}
            \gph\mathcal{F}_x &= \{(t,y)\in [0,T]\times \H : y\in \proj_{F(t,x)}^{\gamma}(0)\}\\
            &= \{(t,y)\in [0,T]\times \H :\|y\|^2< d(0,F(t,x))^2 + \gamma\text{ and } y\in F(t,x)\}\\
            &=\gph F(\cdot,x)\cap \{(t,y)\in [0,T]\times \H :\|y\|^2< d(0,F(t,x))^2 + \gamma\}.
        \end{aligned}
    \end{equation*} 
    Hence, $\gph\mathcal{F}_x$ is a measurable set. Consequently, $\mathcal{F}_x$ has a measurable selection (see \cite[Theorem 6.3.20]{MR2527754}). Denoting by $t\mapsto f(t,x)$ such selection, we obtain the result.
\end{proof}
Now, we establish the main properties of the proposed algorithm.

\begin{theorem}\label{properties_alg}
    Assume, in addition to \ref{FixH1}, \ref{FixH2} and \ref{FixH3}, that $C\colon [0,T]\tto \H$ is a set-valued map with nonempty and closed values such that 
    \begin{equation}\label{lipschitz_hausdorff}
  d_H(C(t),C(s))\leq L_C|t-s| \text{ for all } t,s\in [0,T].       
    \end{equation}
    Then, the sequence of functions $(x_n\colon [0,T]\to \H)_{n\in\mathbb{N}}$ generated by the numerical scheme \eqref{Prox-algorithm-xn} and \eqref{Prox-algorithm} satisfies the following properties:
    \begin{enumerate}[ref=\ref{properties_alg}-(\alph*)]
        \item[$(a)$] \label{prop1_a} There are nonnegative constants $K_1,K_2,K_3, K_4,K_5$ such that for all $n\in\mathbb{N}$ and $t\in [0,T]$:
            \begin{enumerate}
                \item[(i)] $d_{C(\theta_n(t))}(x_n(\delta_n(t)) + \int_{\delta_n(t)}^{\theta_n(t)}f(s,x_n(\delta_n(t)))ds)\leq (L_C + h(x(\delta_n(t)))+\sqrt{\gamma})\mu_n.$ 
                \item[(ii)] $\|x_n(\theta_n(t)) - x_0\|\leq K_1.$ 
                \item[(iii)] $\|x_n(t) \|\leq K_2.$
                \item[(iv)] $\|x_n(\theta_n(t)) - x_n(\delta_n(t))\|\leq K_3\mu_n + \sqrt{\epsilon_n}.$
                \item[(v)] $\|x_n(t)-x_n(\theta_n(t))\|\leq K_4\mu_n + 2\sqrt{\epsilon_n}$.
            \end{enumerate}
        \item[$(b)$] \label{prop1_b} There exists $K_5>0$ such that for all $t\in [0,T]$ and $m,n\in \mathbb{N}$ we have 
        \begin{equation*}
            d_{C(\theta_n(t))}(x_m(t))\leq K_5\mu_m + L_C\mu_n + 2\sqrt{\epsilon_m}.
        \end{equation*}
        \item[$(c)$] \label{prop1_c} There exists $K_6>0$ such that for all $n\in\mathbb{N}$ and almost all $t\in [0,T]$, $\|\dot{x}_n(t)\|\leq K_6$.
        \item[$(d)$] \label{prop1_d} For all $n\in\mathbb{N}$ and $k\in\{0,1,...,n-1\}$, there is $v_{k+1}^n\in C(t_{k+1}^n)$ such that for all $t\in ]t_k^n,t_{k+1}^n[$:
        \begin{equation}\label{inclusion1}
            \dot{x}_n(t)\in -\frac{\lambda_n(t)}{\mu_n}\partial_P d_{C(\theta_n(t))}(v_{k+1}^n) + f(t,x_n(\delta_n(t))) +\frac{3\sqrt{\epsilon_n}}{\mu_n}\mathbb{B},
        \end{equation}
        where $\lambda_n(t) = 4\sqrt{\epsilon_n} + (L_C + h(x(\delta_n(t)))+\sqrt{\gamma})\mu_n$.\\ Moreover, $\|v_{k+1}^n-x_n(\theta_n(t))\|<2\sqrt{\epsilon_n}$.
    \end{enumerate}
\end{theorem}

\begin{proof}
    $(a)$: Set $\mu_n := T/n$ and let $(\epsilon_n)$ be a sequence of nonnegative numbers such that $\epsilon_n/\mu_n^2\to 0$. We define $\mathfrak{c} := \sup_{n\in\mathbb{N}} \frac{\sqrt{\epsilon_n}}{\mu_n}$. We denote by $L_h$ the Lipschitz constant of $h$. For all $t\in [0,T]$ and $n\in\N$, we define $\tau_n(t) := x_n(\delta_n(t)) + \int_{\delta_n(t)}^{\theta_n(t)}f(s,x_n(\delta_n(t)))ds$. Since  $f(t,x_n(\delta_n(t)))\in \proj_{F(t,x_n(\delta_n(t)))}^{\gamma}(0)$ we obtain that
    \begin{equation*}
        \begin{aligned}
            d_{C(\theta_n(t))}(\tau_n(t)) &\leq d_{C(\theta_n(t))}(x_n(\delta_n(t))) + \left\| \int_{\delta_n(t)}^{\theta_n(t)}f(s,x_n(\delta_n(t)))ds \right\|\\
            &\leq L_C\mu_n + \int_{\delta_n(t)}^{\theta_n(t)}\|f(s,x_n(\delta_n(t)))\|ds\\
            &\leq L_C\mu_n + \int_{\delta_n(t)}^{\theta_n(t)}(h(x_n(\delta_n(t)))+\sqrt{\gamma})ds\\
            &\leq (L_C + h(x_n(\delta_n(t)))+\sqrt{\gamma})\mu_n,
        \end{aligned}
    \end{equation*}
    which proves $(i)$.   Moreover, since $x_n(\theta_n(t))\in \proj_{C(\theta_n(t))}^{\epsilon_n}(\tau_n(t))$, we get that
    \begin{equation}\label{cota_1}
        \begin{aligned}
            \| x_n(\theta_n(t)) - \tau_n(t) \| &\leq d_{C(\theta_n(t))}(\tau_n(t)) + \sqrt{\epsilon_n}\\
            &\leq (L_C + h(x_n(\delta_n(t)))+\sqrt{\gamma})\mu_n + \sqrt{\epsilon_n},
        \end{aligned}
    \end{equation}
    which yields
    \begin{equation}\label{cota_2}
        \begin{aligned}
            \| x_n(\theta_n(t)) - x_n(\delta_n(t)) \| &\leq (L_C + 2h(x_n(\delta_n(t))) + 2\sqrt{\gamma})\mu_n + \sqrt{\epsilon_n}\\
            &\leq (L_C + 2h(x_0) + 2\sqrt{\gamma} + 2L_h\|x_n(\delta_n(t))-x_0\|)\mu_n \\
            &+ \sqrt{\epsilon_n}.
        \end{aligned}
    \end{equation}
    Hence, for all $t\in [0,T]$
    \begin{equation*}
        \begin{aligned}
            \| x_n(\theta_n(t)) - x_0 \|
            &\leq (1+2L_h\mu_n)\|x_n(\delta_n(t))-x_0\|\\
            &+(L_C + 2h(x_0) + 2\sqrt{\gamma})\mu_n + \sqrt{\epsilon_n}.
        \end{aligned}
    \end{equation*}
    The above inequality means that for all $k\in \{0,1,...,n-1\}$:
    \begin{equation*}
        \begin{aligned}
            \| x_{k+1}^n - x_0 \|
            &\leq (1+2L_h\mu_n)\|x_k^n-x_0\|+(L_C + 2h(x_0)+2\sqrt{\gamma})\mu_n + \sqrt{\epsilon_n}.
        \end{aligned}
    \end{equation*}
    Then, by \cite[p.~183]{Clarke1998}, we obtain that for all $k\in \{0,...,n-1\}$
    \begin{equation}\label{cota_3}
        \begin{aligned}
            \| x_{k+1}^n - x_0 \| &\leq (k+1)((L_C + 2h(x_0)+2\sqrt{\gamma})\mu_n + \sqrt{\epsilon_n})\exp(2L_h(k+1)\mu_n)\\
            &\leq T(L_C + 2h(x_0) +\sqrt{\gamma} + \mathfrak{c})\exp(2L_h T)=:K_1.
        \end{aligned}
    \end{equation}
    which proves $(ii)$.\\
    $(iii)$: By definition of $x_n$,  for $t\in ]t_k^n,t_{k+1}^n]$ and $k\in \{0,1...,n-1\}$, using \eqref{Prox-algorithm-xn}
    \begin{equation*}
        \begin{aligned}
            \| x_n(t) \| &\leq \|x_k^n\| + \| x_{k+1}^n - \tau_n(t) \| + \int_{t_k^n}^t\|f(s,x_k^n)\|ds\\
            &\leq K_1 + \|x_0\| + (L_C +\sqrt{\gamma} +h(x_k^n))\mu_n + \sqrt{\epsilon_n} + (h(x_k^n)+\sqrt{\gamma})\mu_n,
        \end{aligned}
    \end{equation*}
    where we have used \eqref{cota_1}. Moreover, it is clear that for $k\in\{0,...,n\}$
    \begin{equation*}
        \begin{aligned}
            h(x_k^n)\leq h(x_0) + L_h\|x_k^n - x_0\|\leq h(x_0) + L_hK_1.
        \end{aligned}
    \end{equation*}
    Therefore, for all $t\in [0,T]$
    \begin{equation*}
        \begin{aligned}
            \|x_n(t)\| &\leq K_1 + \|x_0\| + (L_C + 2(h(x_0) + L_hK_1+\sqrt{\gamma}))\mu_n + \sqrt{\epsilon_n}\\
            &\leq K_1 + \|x_0\| + T(L_C+2(h(x_0) + L_hK_1+\sqrt{\gamma}) + \mathfrak{c}) =: K_2,
        \end{aligned}
    \end{equation*}
    which proves $(iii)$.\\
    $(iv)$: From \eqref{cota_2} and \eqref{cota_3} it is easy to see that there exists $K_3>0$ such that for all $n\in\N$ and $t\in[0,T]$: $\| x_n(\theta_n(t)) - x_n(\delta_n(t)) \|\leq K_3\mu_n + \sqrt{\epsilon_n}$.\\
    $(v)$: To conclude this part, we consider $t\in ]t_k^n,t_{k+1}^n]$ for some $k\in \{0,1,...,n-1\}$. Then $x_n(\theta_n(t)) = x_{k+1}^n$ and also
    \begin{equation*}
        \begin{aligned}
            \|x_n(\theta_n(t))-x_n(t)\| \leq & \|x_{k+1}^n - x_k^n\| + \|x_{k+1}^n - \tau_n(t)\| + \int_{t_k^n}^t \|f(s,x_k^n)\|ds\\
            \leq & K_3\mu_n + \sqrt{\epsilon_n} +  (L_C + \sqrt{\gamma}+ h(x_0) + L_h K_1)\mu_n + \sqrt{\epsilon_n}\\
            &+ \mu_n (h(x_k^n)+\sqrt{\gamma})\\
            \leq & (\underbrace{K_3 + L_C + 2(h(x_0) + L_hK_1)+2\sqrt{\gamma}}_{=:K_4})\mu_n + 2\sqrt{\epsilon_n},
        \end{aligned}
    \end{equation*}
    and we conclude this first part.\\
    $(b)$: Let $m,n\in\N$ and $t\in [0,T]$, then 
    \begin{equation*}
        \begin{aligned}
            d_{C(\theta_n(t))}(x_m(t)) &\leq d_{C(\theta_n(t))}(x_m(\theta_m(t))) + \|x_m(\theta_m(t))-x_m(t)\|\\
            &\leq d_H(C(\theta_n(t)),C(\theta_m(t))) + K_4\mu_m + 2\sqrt{\epsilon_m}\\
            &\leq L_C|\theta_n(t) - \theta_m(t)| + K_4\mu_m + 2\sqrt{\epsilon_m}\\
            &\leq L_C(\mu_n + \mu_m) + K_4\mu_m + 2\sqrt{\epsilon_m}
        \end{aligned}
    \end{equation*}
    where we have used (v).
    Hence, by setting $K_5 := K_4 + L_C$ we prove (b).\\ 
    $(c)$: Let $n\in\N$, $k\in\{0,1,...,n-1\}$ and $t\in ]t_k^n,t_{k+1}^n]$. Then,
    \begin{equation*}
        \begin{aligned}
            \|\dot{x}_n(t)\| &= \left\|  \frac{1}{\mu_n}\left(x_{k+1}^n-x_k^n-\int_{t_{k}^n}^{t_{k+1}^n}f(s,x_k^n)ds\right) + f(t,x_k^n)\right\|\\
            &\leq \frac{1}{\mu_n}\|x_n(\theta_n(t))-\tau_n(t)\| + \|f(t,x_k^n)\|\\
            &\leq \frac{1}{\mu_n}((L_C + h(x_k^n)+\sqrt{\gamma})\mu_n+\sqrt{\epsilon_n}) + h(x_k^n) + \sqrt{\gamma}\\
            &\leq \frac{\sqrt{\epsilon_n}}{\mu_n} + L_C + 2(h(x_0) + L_hK_1 + \sqrt{\gamma})\\
            &\leq \mathfrak{c} + L_C + 2(h(x_0) + L_hK_1+ \sqrt{\gamma}) =: K_6,
        \end{aligned}
    \end{equation*}
    which proves $(c)$.\\
    $(d)$: Fix $k\in \{0,1,...,n-1\}$ and $t\in ]t_k^n,t_{k+1}^n[$. Then, $x_{k+1}^n\in \proj_{C(t_{k+1}^n)}^{\epsilon_n}(\tau_n(t))$. Hence,  by Lemma \ref{proximal-lemma}, there exists $v_{k+1}^n\in C(t_{k+1}^n)$ such that $\|x_{k+1}-v_{k+1}^n\|<2\sqrt{\epsilon_n}$ and 
    \begin{equation*}
        \begin{aligned}
            \tau_n(t)-x_{k+1}^n\in \alpha_n(t)\partial_P d_{C(t_{k+1}^n)}(v_{k+1}^n) + 3\sqrt{\epsilon_n}\mathbb{B}, \ \forall t\in ]t_k^n,t_{k+1}^n[,
        \end{aligned}
    \end{equation*}
    where $\alpha_n(t) = 4\sqrt{\epsilon_n} + d_{C(\theta_n(t))}(\tau_n(t))$. By virtue of $(i)$,  $$\alpha_n(t)\leq 4\sqrt{\epsilon_n} + (L_C + h(x(\delta_n(t)))+\sqrt{\gamma})\mu_n=:\lambda_n(t).$$ 
    Then, for all $t\in ]t_k^n,t_{k+1}^n[$  
    \begin{equation*}
        -\mu_n(\dot{x}_n(t)-f(t,x_k^n))\in \lambda_n(t)\partial_P d_{C(t_{k+1}^n)}(v_{k+1}^n) + 3\sqrt{\epsilon_n}\mathbb{B},
    \end{equation*}
    which implies that $t\in ]t_k^n,t_{k+1}^n[$
    \begin{equation*}
        \dot{x}_n(t)\in -\frac{\lambda_n(t)}{\mu_n}\partial_P d_{C(t_{k+1}^n)}(v_{k+1}^n) + f(t,x_k^n) + \frac{3\sqrt{\epsilon_n}}{\mu_n}\mathbb{B}.
    \end{equation*}
\end{proof}

\section{Prox-Regular Case}

In this section, we will study the algorithm under the assumption of uniform prox-regularity of the moving set. The classic catching-up algorithm in this framework was studied  \cite{MR2159846}, where the existence of solutions for \eqref{Fixed-sweep} was established for a set-valued map $F$ taking values in a fixed compact set.

\begin{theorem}\label{thm-prox}
Suppose, in addition to the assumptions of Theorem \ref{properties_alg}, that $C(t)$ is $\rho$-uniformly prox-regular  for all $t\in [0,T]$, and there exists a nonnegative integrable function $k$ such that for all $t\in [0,T]$ and $x,x'\in\H$
    \begin{equation}\label{monotonicity}
        \langle y-y',x-x'\rangle\leq k(t)\|x-x'\|^2, \ \forall y\in F(t,x), \forall y'\in F(t,x').
    \end{equation}
    Then, the sequence of functions $(x_n)_n$ generated by the algorithm \eqref{Prox-algorithm-xn} and \eqref{Prox-algorithm} converges uniformly to an absolutely continuous function $x$, which is the unique solution of \eqref{Fixed-sweep}.
\end{theorem}
\begin{proof}
    Consider $m,n\in\N$ with $m\geq n$ big enough such that for all $t\in [0,T]$, $d_{C(\theta_n(t))}(x_m(t))<\rho$, this can be guaranteed by Theorem \ref{prop1_b}. Then,  for a.e. $t\in [0,T]$ 
    \begin{equation*}
        \frac{d}{dt}\left(\frac{1}{2}\|x_n(t)-x_m(t)\|^2\right) = \langle \dot{x}_n(t)-\dot{x}_m(t), x_n(t)-x_m(t)\rangle.
    \end{equation*}
    Let $t\in [0,T]$ where the above equality holds. Let $k, j \in\{0,1,...,n-1\}$ such that $t\in ]t_k^n,t_{k+1}^n]$ and $t\in ]t_j^m,t_{j+1}^m]$. On the one hand,  we have that
    \begin{equation}\label{bound_1}
        \begin{aligned}
            \langle \dot{x}_n(t)-\dot{x}_m(t), x_n(t)-x_m(t)\rangle  = & \ \langle \dot{x}_n(t)-\dot{x}_m(t),x_n(t)-x_{k+1}^n\rangle\\
            &+ \langle \dot{x}_n(t)-\dot{x}_m(t),x_{k+1}^n-v_{k+1}^n\rangle\\
            & + \langle \dot{x}_n(t)-\dot{x}_m(t),v_{k+1}^n -v_{j+1}^m\rangle\\
            &+\langle \dot{x}_n(t)-\dot{x}_m(t), v_{j+1}^m- x_{j+1}^m \rangle \\
            &+ \langle\dot{x}_n(t)-\dot{x}_m(t),x_{j+1}^m- x_m(t)\rangle\\
            \leq  & \ 2K_6(K_4(\mu_n +\mu_m) +4(\sqrt{\epsilon_n} + \sqrt{\epsilon_m}))\\
         &+\langle \dot{x}_n(t)-\dot{x}_m(t),v_{k+1}^n -v_{j+1}^m\rangle   
        \end{aligned}
    \end{equation}
where $v_{k+1}^n\in C(t_{k+1}^n)$ and $v_{j+1}^m\in C(t_{j+1}^m)$ are the given in Theorem \ref{prop1_d}. We can see that 
\begin{equation*}
    \begin{aligned}
 \max\{d_{C(t_{k+1}^n)}(v_{j+1}^m),d_{C(t_{j+1}^m)}(v_{k+1}^n)\} &\leq d_H(C(t_{j+1}^m),C(t_{k+1}^n))\\
    &\leq L_C|t_{j+1}^m-t_{k+1}^n|\leq L_C(\mu_n+\mu_m).        
    \end{aligned}
\end{equation*}
From now, $m,n\in\N$ are big enough such that $L_C(\mu_n+\mu_m)<\frac{\rho}{2}$. Moreover, as $h$ is $L_h$-Lipschitz, we have that for all $p\in\N$, $i\in \{0,1,...,p\}$ and $t\in[0,T]$
$$\|f(t,x_i^p)\|\leq h(x_i^p)+\sqrt{\gamma}\leq h(x_0) + L_hK_1+\sqrt{\gamma}=:\alpha. $$ From the other hand, using \eqref{inclusion1} and Proposition \ref{prox_reg_prop} we have that 
    \begin{equation*}
        \begin{aligned}
        &\frac{1}{\digamma}\max\{\left\langle \zeta_n-\dot{x}_n(t),v_{j+1}^m-v_{k+1}^n\right\rangle, \left\langle \zeta_m-\dot{x}_m(t),v_{k+1}^n-v_{j+1}^m\right\rangle\}\\
        \leq & \frac{2}{\rho}\|v_{k+1}^n-v_{j+1}^m\|^2+L_C(\mu_n+\mu_m)
        \end{aligned}
    \end{equation*}
    % \begin{equation*}
    %     \begin{aligned}
    %     &\frac{1}{\digamma}\left\langle \zeta_m-\dot{x}_m(t),v_{k+1}^n-v_{j+1}^m\right\rangle\\
    %     \leq & \frac{2}{\rho}\|v_{k+1}^n-v_{j+1}^m\|^2+L_C(\mu_n+\mu_m),
    %     \end{aligned}
    % \end{equation*}
    where $\xi_n,\xi_m\in\mathbb{B}$, $\digamma:=\sup\{\frac{\lambda_\ell(t)}{\mu_\ell}:t\in [0,T],\ell\in\N\}$ and $\zeta_i := f(t,x_i(\delta_i(t)))+\frac{3\sqrt{\epsilon_i}}{\mu_i}\xi_i$ for $i\in \{n,m\}$. Therefore, we  have that
\begin{equation*}
    \begin{aligned}
&\langle \dot{x}_n(t)-\dot{x}_m(t),v_{k+1}^n -v_{j+1}^m\rangle\\
    &\hspace{21mm}=   \langle \dot{x}_n(t)-\zeta_n,v_{k+1}^n -v_{j+1}^m\rangle+\langle  \zeta_n-\zeta_m,v_{k+1}^n -v_{j+1}^m\rangle \\
    &\hspace{21mm}+ \langle  \zeta_m-\dot x_m(t),v_{k+1}^n -v_{j+1}^m\rangle \\
   &\hspace{21mm} \leq   2\digamma(\frac{2}{\rho}\|v_{k+1}^n-v_{j+1}^m\|^2 + L_C(\mu_n+\mu_m))+\langle  \zeta_n-\zeta_m,v_{k+1}^n -v_{j+1}^m\rangle\\
    &\hspace{21mm} \leq  \frac{4\digamma}{\rho}(\|x_n(t)-x_m(t)\| + 3(\sqrt{\epsilon_n}+\sqrt{\epsilon_m}) + K_4(\mu_n+\mu_m))^2 \\
    &\hspace{21mm}+ 2\digamma L_C(\mu_n+\mu_m)
    +\langle  \zeta_n-\zeta_m,v_{k+1}^n -v_{j+1}^m\rangle.
    \end{aligned}
\end{equation*}
Moreover, using Theorem \ref{properties_alg} and property \eqref{monotonicity},
\begin{equation*}
    \begin{aligned}
        &  \langle  \zeta_n-\zeta_m,v_{k+1}^n -v_{j+1}^m\rangle\\
        &\hspace{21mm} =   \langle  f(t,x_n(\delta_n(t)))-f(t,x_m(\delta_m(t))),x_n(\delta_n(t))-x_m(\delta_m(t))\rangle\\
    &\hspace{21mm} + \langle  f(t,x_n(\delta_n(t)))-f(t,x_m(\delta_m(t))),v_{k+1}^n-x_{k+1}^n\rangle\\
    &\hspace{21mm}+ \langle  f(t,x_n(\delta_n(t)))-f(t,x_m(\delta_m(t))),x_{k+1}^n-x_k^n\rangle\\
    &\hspace{21mm} + \langle  f(t,x_n(\delta_n(t)))-f(t,x_m(\delta_m(t))),x_j^m-x_{j+1}^m\rangle \\
    &\hspace{21mm}+ \langle  f(t,x_n(\delta_n(t)))-f(t,x_m(\delta_m(t))),x_{j+1}^m-v_{j+1}^m\rangle \\
    &\hspace{21mm} + \frac{3\sqrt{\epsilon_n}}{\mu_n}\langle\xi_n,v_{k+1}^n-v_{j+1}^m\rangle + \frac{3\sqrt{\epsilon_m}}{\mu_m}\langle\xi_m,v_{j+1}^m-v_{k+1}^n\rangle\\
    &\hspace{21mm}\leq  \ k(t)\|x_n(\delta_n(t))-x_m(\delta_m(t))\|^2\\
    &\hspace{21mm} + 2\alpha(3(\sqrt{\epsilon_n}+\sqrt{\epsilon_m})+K_3(\mu_n+ \mu_m)) \\
    &\hspace{21mm} + \frac{3\sqrt{\epsilon_n}}{\mu_n}\|v_{k+1}^n-v_{j+1}^m\| + \frac{3\sqrt{\epsilon_m}}{\mu_m}\|v_{j+1}^m-v_{k+1}^n\|\\
    &\hspace{21mm}\leq   k(t)(\|x_n(t)-x_m(t)\| + 3(\sqrt{\epsilon_n}+\sqrt{\epsilon_m})+(K_3+K_4)(\mu_n+\mu_m))^2\\
    &\hspace{21mm} + 2\alpha(3(\sqrt{\epsilon_n}+\sqrt{\epsilon_m})+K_3(\mu_n+ \mu_m)) \\
    &\hspace{21mm} + 6\left(\frac{\sqrt{\epsilon_n}}{\mu_n}+\frac{\sqrt{\epsilon_m}}{\mu_m}\right)(\sqrt{\epsilon_n}+\sqrt{\epsilon_m}+K_2).
    \end{aligned}
\end{equation*}
 \noindent These two inequalities and \eqref{bound_1}  yield 
\begin{equation*}
    \begin{aligned}
        &\frac{d}{dt}\|x_n(t)-x_m(t)\|^2\\
        &\hspace{10mm}\leq  \ 4\left(\frac{4\digamma}{\rho}+k(t)\right)\|x_n(t)-x_m(t)\|^2+4\alpha(3(\sqrt{\epsilon_n}+\sqrt{\epsilon_m})+K_3(\mu_n+ \mu_m))\\
        &\hspace{10mm}+4\digamma L_C(\mu_n+\mu_m)+12\left(\frac{\sqrt{\epsilon_n}}{\mu_n}+\frac{\sqrt{\epsilon_m}}{\mu_m}\right)(\sqrt{\epsilon_n}+\sqrt{\epsilon_m}+K_2)\\
        &\hspace{10mm}+\frac{16\digamma}{\rho}(3(\sqrt{\epsilon_n}+\sqrt{\epsilon_m}) + K_4(\mu_n+\mu_m))^2\\
        &\hspace{10mm}+4k(t)(3(\sqrt{\epsilon_n}+\sqrt{\epsilon_m})+(K_3+K_4)(\mu_n+\mu_m))^2.
    \end{aligned}
\end{equation*} 
Hence, using Gronwall's inequality, we have for all $t\in [0,T]$ and $n,m$ big enough:
\begin{equation}\label{xn-xm}
    \|x_n(t)-x_m(t)\|^2\leq A_{m,n}\exp\left(\frac{16\digamma}{\rho}T+4\int_0^T k(s)ds\right),
\end{equation}
where,
\begin{equation*}
    \begin{aligned}
        A_{m,n} &=\ 4\alpha T(3(\sqrt{\epsilon_n}+\sqrt{\epsilon_m})+K_3(\mu_n+ \mu_m))\\
        &+4T\digamma L_C(\mu_n+\mu_m)+12T\left(\frac{\sqrt{\epsilon_n}}{\mu_n}+\frac{\sqrt{\epsilon_m}}{\mu_m}\right)(\sqrt{\epsilon_n}+\sqrt{\epsilon_m}+K_2)\\
        &+\frac{16T\digamma}{\rho}(3(\sqrt{\epsilon_n}+\sqrt{\epsilon_m}) + K_4(\mu_n+\mu_m))^2\\
        &+4\|k\|_1(3(\sqrt{\epsilon_n}+\sqrt{\epsilon_m})+(K_3+K_4)(\mu_n+\mu_m))^2.
    \end{aligned}
\end{equation*}
Since $A_{m,n}$ goes to $0$ when $m,n\to\infty$, it shows that $(x_n)_{n\in\N}$ is a Cauchy sequence in the space of continuous functions with the uniform convergence. Therefore, it converges uniformly to some continuous function $x\colon [0,T]\to \H$. It remains to check that $x$ is absolutely continuous, and it is the unique solution of \eqref{Fixed-sweep}. First of all, by Theorem \ref{prop1_c} and \cite[Lemma 2.2]{MR3626639}, $x$ is absolutely continuous and there is a subsequence of $(\dot{x}_{n})$ which converges weakly in $L^1([0,T];\H)$ to $\dot{x}$. So, without relabeling, we have $\dot{x}_n\rightharpoonup \dot{x}$ in $L^1([0,T];\H)$. On the other hand,  using Theorem \ref{prop1_d} and defining $v_n(t) := v_{k+1}^n$ for $t\in ]t_k^n,t_{k+1}^n]$ we have
\begin{equation*}
    \begin{aligned}
        \dot{x}_n(t) &\in -\frac{\lambda_n(t)}{\mu_n}\partial_P d_{C(\theta_n(t))}(v_n(t)) + f(t,x_n(\delta_n(t))) +\frac{3\sqrt{\epsilon_n}}{\mu_n}\mathbb{B}\\
        &\in -\kappa_1\partial d_{C(\theta_n(t))}(v_n(t)) + \kappa_2\mathbb{B}\cap F(t,x_n(\delta_n(t))) + \frac{3\sqrt{\epsilon_n}}{\mu_n}\mathbb{B}.
    \end{aligned}
\end{equation*}
where, by Theorem \ref{properties_alg},  $\kappa_1$ and $\kappa_2$ are nonnegative numbers which do not depend of $n\in\N$ and $t\in [0,T]$. We also have $v_n\to x$, $\theta_n\to \text{Id}_{[0,T]}$ and $\delta_n\to \text{Id}_{[0,T]}$ uniformly. Theorem \ref{prop1_b} ensures that $x(t)\in C(t)$ for all $t\in [0,T]$. By Mazur's Lemma, there is a sequence $(y_j)$ such that for all $n$, $y_n\in \text{co}(\dot{x}_k:k\geq n)$ and $(y_n)$ converges strongly to $\dot{x}$ in $L^1([0,T];\H)$. That is to say $$y_n(t)\in \text{co}\left( -\kappa_1\partial d_{C(\theta_k(t))}(v_k(t)) + \kappa_2\mathbb{B}\cap F(t,x_k(\delta_k(t))) +\frac{3\sqrt{\epsilon_k}}{\mu_k}\mathbb{B} :k\geq n\right).$$ Hence, there exists $(y_{n_j})$ which converges to $\dot{x}$ almost everywhere in $[0,T]$. Then, by virtue of Lemma \ref{upper-semicontinuity}, \ref{FixH1} and Lemma \ref{lemma_tech_1}, we obtain that $$\dot{x}(t)\in -\kappa_1\partial d_{C(t)}(x(t)) +\kappa_2\mathbb{B}\cap F(t,x(t)) \textrm{ for a.e. } t\in [0,T].$$
Finally, since $\partial d_{C(t)}(x(t))\subset N(C(t);x(t))$ for all $t\in[0,T]$, we have $x$ is the solution of \eqref{Fixed-sweep}.
\end{proof}
\begin{remark}
    The property required for $F$ in \eqref{monotonicity} is a classical monotonicity assumption in the theory of existence of solutions for differential inclusions (see, e.g., \cite[Theorem 10.5]{MR1189795}).
\end{remark}

\begin{remark}[Rate of convergence]
    In the precedent proof, we have established the following estimation:
    \begin{equation*}
    \|x_n(t)-x_m(t)\|^2\leq A_{m,n}\exp\left(\frac{16\digamma}{\rho}T+4\int_0^T k(s)ds\right)
\end{equation*}
for $m,n$ such that $\mu_n+\mu_m<\frac{\rho}{2L_C}$. Hence, by letting $m\to \infty$, we obtain that 
\begin{equation*}
    \|x_n(t)-x(t)\|^2\leq A_n\exp\left(\frac{16\digamma}{\rho}T+4\int_0^T k(s)ds\right) \textrm{ for all } n>\frac{2L_C T}{\rho},
\end{equation*}
where $$A_n := \lim_{m\to\infty} A_{m,n}\leq D(\sqrt{\epsilon_n}  + \mu_n + \frac{\sqrt{\epsilon_n}}{\mu_n}),$$ where $D$ is a nonnegative constant. Hence, the above estimation provides a rate of convergence for our scheme.  
\end{remark}

\section{Subsmooth case}
In this section, we study the sweeping process \eqref{Fixed-sweep} in a more general setting than the uniformly prox-regular case. We now assume $(C(t))_{t\in [0,T]}$ is a equi-uniformly subsmooth family. The classical catching-up algorithm was studied in \cite{MR3574145} under this framework. In this case, we make the assumption about the ball compactness of the moving sets. We will see that our algorithm allows us to prove the existence of a solution, but we only ensure that a subsequence converges to this solution, which is expected due to the lack of uniqueness of solutions in this case.
\begin{theorem}
Suppose, in addition to assumptions of theorem  \ref{properties_alg}, that the family $(C(t))_{t\in [0,T]}$ is equi-uniformly subsmooth and the set $C(t)$ are ball-compact for all $t\in [0,T]$. Then, the sequence of continuous functions $(x_n)_n$ generated by algorithm  \eqref{Prox-algorithm-xn} and \eqref{Prox-algorithm} converges uniformly (up to a subsequence) to an absolutely continuous function $x$, which is a solution of \eqref{Fixed-sweep}.
\end{theorem}

\begin{proof}
    From Theorem \ref{prop1_d} we have for all $n\in\N$ and $k\in \{0,\ldots,n-1\}$, there is $v_{k+1}^n\in C(t_{k+1}^n)$ such that $\|v_{k+1}^n-x_{k+1}^n\|<2\sqrt{\epsilon_n}$ and for all $t\in ]t_k^n,t_{k+1}^n]$:
\begin{equation*}
    \dot{x}_n(t)\in -\frac{\lambda_n(t)}{\mu_n}\partial_P d_{C(\theta_n(t))}(v_{k+1}^n) + f(t,x_n(\delta_n(t))) +\frac{3\sqrt{\epsilon_n}}{\mu_n}\mathbb{B}.
\end{equation*}
where $\lambda_n(t) = 4\sqrt{\epsilon_n} + (L_C + h(x(\delta_n(t)))+\sqrt{\gamma})\mu_n$. As $h$ is $L_h$-Lipschitz it follows that $$\lambda_n(t)\leq (4\mathfrak{c}+L_C+h(x_0)+\sqrt{\gamma} + L_hK_1)\mu_n.$$ Defining $v_n(t):=v_{k+1}^n$ on $]t_k^n,t_{k+1}^n]$, then for all $n\in\N$ and almost all $t\in [0,T]$
\begin{equation}\label{inclusion6}
    \begin{aligned}
        \dot{x}_n(t) &\in -M\partial_P d_{C(\theta_n(t))}(v_n(t)) + f(t,x_n(\delta_n(t))) +\frac{3\sqrt{\epsilon_n}}{\mu_n}\mathbb{B}\\
        &\in -M\partial d_{C(\theta_n(t))}(v_n(t)) + M\mathbb{B}\cap F(t,x_n(\delta_n(t))) +\frac{3\sqrt{\epsilon_n}}{\mu_n}\mathbb{B}.     
    \end{aligned}
\end{equation}
where $M:= 4\mathfrak{c}+L_C+h(x_0) + L_hK_1+\sqrt{\gamma}$.
 Moreover, by Theorem \ref{prop1_b},  we have 
\begin{equation}\label{cota2222}
\begin{aligned}
d_{C(t)}(x_n(t))\leq d_{C(\theta_n(t))}(x_n(t)) + L_C\mu_n\leq   (K_5+2L_C)\mu_n+2\sqrt{\varepsilon_n}.
\end{aligned}
\end{equation}
for all  $t\in [0,T]$. \newline
Next, fix $t\in [0,T]$ and define $K(t):=\{x_n(t) : n\in \mathbb{N}\}$. We claim that $K(t)$ is relatively compact. Indeed, let $x_m(t) \in K(t)$ and take $y_m(t)\in \operatorname{Proj}_{C(t)}(x_m(t))$ (the projection exists due to the ball compactness of $C(t)$ and the boundedness of $K(t)$). Moreover, according to \eqref{cota2222} and Theorem \ref{prop1_a},
\begin{equation*}
\begin{aligned}
\Vert y_n(t)\Vert & \leq d_{C(t)}(x_n(t))+\Vert x_n(t)\Vert \leq  (K_5+2L_C)\mu_n+2\sqrt{\varepsilon_n}+K_2.
\end{aligned}
\end{equation*}
This entails that $y_n(t)\in C(t)\cap R\, \mathbb{B}$ for all $n\in\mathbb{N}$ for some $R>0$. Thus, by the ball compactness of $C(t)$, there exists a subsequence $(y_{m_k}(t))_{m_k}$ of $(y_m(t))_m$ converging to some $y(t)$ as $k\to +\infty$. Then, 
\begin{equation*}
\begin{aligned}
\Vert x_{m_k}(t)-y(t)\Vert &\leq d_{C(t)}(x_{m_k}(t))+\Vert y_{m_k}(t)-y(t)\Vert \\
&\leq (K_5+2L_C)\mu_{m_k}+2\sqrt{\varepsilon_{m_k}}+\Vert y_{m_k}(t)-y(t)\Vert,
\end{aligned}
\end{equation*}
which implies that $K(t)$ is relatively compact. Moreover, it is not difficult to see by Theorem \ref{prop1_c} that $K:=(x_n)$ is equicontinuous.  Therefore, by virtue of Theorem \ref{prop1_c}, Arzela-Ascoli's and Dunford-Pettis's Theorems, we obtain the existence of a Lipschitz function $x$ and a subsequence $(x_j)_j$ of $(x_n)_n$ such that
\begin{enumerate}[label=(\roman{*})]
\item $(x_j)$ converges uniformly to $x$ on $[0,T]$.
\item  $\dot{x}_j\rightharpoonup \dot{x}$ in $L^1\left([0,T];\H\right)$.
\item $x_j(\theta_j(t))\to x(t)$ for all $t\in [0,T]$.
\item $x_j(\delta_j(t))\to x(t)$ for all $t\in [0,T]$.
\item $v_j(t)\to x(t)$ for all $t\in [0,T]$.
\end{enumerate}
From \eqref{cota2222} it is clear that $x(t)\in C(t)$ for all $t\in [0,T]$. By Mazur's Lemma, there is a sequence $(y_j)$ such that for all $j$, $y_j\in \text{co}(\dot{x}_k:k\geq j)$ and $(y_j)$ converges strongly to $\dot{x}$ in $L^1([0,T];\H)$. That is to say $$y_j(t)\in \text{co}\left( -M\partial d_{C(\theta_n(t))}(v_n(t)) + M\mathbb{B}\cap F(t,x_n(\delta_n(t))) +\frac{3\sqrt{\epsilon_n}}{\mu_n}\mathbb{B} :n\geq j\right).$$ On the other hand, there exists $(y_{n_j})$ which converges to $\dot{x}$ almost everywhere in $[0,T]$. Then, using Lemma \ref{upper-semicontinuity}, Lemma \ref{lemma_tech_1} and \ref{FixH1}, we have $$\dot{x}(t)\in -M\partial d_{C(t)}(x(t)) +M\mathbb{B}\cap F(t,x(t)) \ \text{a.e.}$$
Finally, since $\partial d_{C(t)}(x(t))\subset N(C(t);x(t))$ for all $t\in [0,T]$, it follows that $x$ is the solution of \eqref{Fixed-sweep}.
\end{proof}

\section{Fixed set}
In this section, we consider a closed and nonempty set $C\subset \H$, and we look for a solution of the particular case of \eqref{Fixed-sweep} given by
\begin{equation}\label{Fixed-sweep-2}
\left\{
\begin{aligned}
\dot{x}(t)&\in -N\left(C;x(t)\right)+F(t,x(t)) & \textrm{ a.e. } t\in [0,T],\\
x(0)&=x_0\in C,
\end{aligned}
\right.
\end{equation}
where $F\colon [0,T]\times \H\tto \H$ is a set-valued map defined as above. The existence of a solution using classical catching up was done in \cite{MR3956966}. Now we use similar ideas to get the existence of a solution using our proposed algorithm. We emphasize that in this case, no regularity of the set $C$ is required.
\begin{theorem}\label{Main-fixed}
Let $C\subset \H$ be a ball-compact set and $F\colon [0,T]\times \H\rightrightarrows \H$ be a set-valued map satisfying \ref{FixH1}, \ref{FixH2} and \ref{FixH3}. Then, for any $x_0\in S$, the sequence of functions $(x_n)_n$ generated by the algorithm \eqref{Prox-algorithm} converges uniformly (up to a subsequence) to a  Lipschitz solution  $x$ of the sweeping process \textup{(\ref{Fixed-sweep-2})} such that
\begin{equation*}
\begin{aligned}
\Vert \dot{x}(t)\Vert  & \leq  2(h(x(t))+\sqrt{\gamma}) & \textrm{ a.e. } t\in [0,T].
\end{aligned}
\end{equation*}

\end{theorem}
\begin{proof} We are going to use the properties of Theorem \ref{properties_alg}, where now we have $L_C = 0$. First of all, from Theorem \ref{prop1_d} we have for all $n\in\N$ and $k\in \{0,1,\ldots,n-1\}$, there is $v_{k+1}^n\in C$ such that $\|v_{k+1}^n-x_{k+1}^n\|<2\sqrt{\epsilon_n}$ and for all $t\in ]t_k^n,t_{k+1}^n]$:
\begin{equation*}
    \dot{x}_n(t)\in -\frac{\lambda_n(t)}{\mu_n}\partial_P d_{C}(v_{k+1}^n) + f(t,x_n(\delta_n(t))) +\frac{3\sqrt{\epsilon_n}}{\mu_n}\mathbb{B}.
\end{equation*}
where $\lambda_n(t) = 4\sqrt{\epsilon_n} + (h(x(\delta_n(t)))+\sqrt{\gamma})\mu_n$. Defining $v_n(t):=v_{k+1}^n$ on $]t_k^n,t_{k+1}^n]$, we get that for all $n\in\N$ and a.e. $t\in [0,T]$
\begin{equation*}
    \begin{aligned}
        \dot{x}_n(t) &\in  -\frac{\lambda_n(t)}{\mu_n}\partial_P d_{C}(v_n(t)) + f(t,x_n(\delta_n(t))) +\frac{3\sqrt{\epsilon_n}}{\mu_n}\mathbb{B}\\
        &\in  -\frac{\lambda_n(t)}{\mu_n}\partial d_{C}(v_n(t))+ (h(t,x_n(\delta_n(t)))+\sqrt{\gamma})\mathbb{B}\cap F(t,x_n(\delta_n(t))) +\frac{3\sqrt{\epsilon_n}}{\mu_n}\mathbb{B}.   
    \end{aligned}
\end{equation*}
 Moreover, by Theorem \ref{prop1_b},  we have 
\begin{equation*}%\label{cota222}
d_C(x_n(t))\leq  K_5\mu_n+2\sqrt{\varepsilon_n} \textrm{ for all } t\in [0,T].
\end{equation*}
Next, fix $t\in [0,T]$ and define $K(t):=\{x_n(t) : n\in \mathbb{N}\}$. We claim that $K(t)$ is relatively compact. Indeed, let $x_m(t) \in K(t)$ and take $y_m(t)\in \operatorname{Proj}_C(x_m(t))$ (the projection exists due to the ball compactness of $C$ and the boundedness of $K(t)$). Moreover, according to the above inequality and Theorem \ref{prop1_a},
\begin{equation*}
\Vert y_n(t)\Vert  \leq d_C(x_n(t))+\Vert x_n(t)\Vert \leq   K_5\mu_n+2\sqrt{\varepsilon_n}+K_2,
\end{equation*}
which entails that $y_n(t)\in C\cap R\, \mathbb{B}$ for all $n\in\mathbb{N}$ for some $R>0$. Thus, by the ball-compactness of $C$, there exists a subsequence $(y_{m_k}(t))_{m_k}$ of $(y_m(t))_m$ converging to some $y(t)$ as $k\to +\infty$. Then, 
\begin{equation*}
\begin{aligned}
\Vert x_{m_k}(t)-y(t)\Vert &\leq d_{C}(x_{m_k}(t))+\Vert y_{m_k}(t)-y(t)\Vert \\
&\leq K_5\mu_{m_k}+2\sqrt{\varepsilon_{m_k}}+\Vert y_{m_k}(t)-y(t)\Vert,
\end{aligned}
\end{equation*}
which implies that $K(t)$ is relatively compact. Moreover, it is not difficult to see by Theorem \ref{prop1_c} that the set $K:=(x_n)$ is equicontinuous.  Therefore, by virtue of Theorem \ref{prop1_c}, Arzela-Ascoli's and Dunford-Pettis's Theorems, we obtain the existence of a Lipschitz function $x$ and a subsequence $(x_j)_j$ of $(x_n)_n$ such that
\begin{enumerate}[label=(\roman{*})]
\item $(x_j)$ converges uniformly to $x$ on $[0,T]$.
\item  $\dot{x}_j\rightharpoonup \dot{x}$ in $L^1\left([0,T];\H\right)$.
\item $x_j(\theta_j(t))\to x(t)$ for all $t\in [0,T]$.
\item $x_j(\delta_j(t))\to x(t)$ for all $t\in [0,T]$.
\item $v_j(t)\to x(t)$ for all $t\in [0,T]$.
\item $x(t)\in C$  for all $t\in [0,T]$.
\end{enumerate} 
By Mazur's Lemma, there is a sequence $(y_j)$ such that for all $j$, $y_j\in \text{co}(\dot{x}_k:k\geq j)$ and $(y_j)$ converges strongly to $\dot{x}$ in $L^1([0,T];\H)$. i.e., $$y_j(t)\in \text{co}\left( -\alpha_n\partial d_{C}(v_n(t)) + \beta_n\mathbb{B}\cap F(t,x_n(\delta_n(t))) +\frac{3\sqrt{\epsilon_n}}{\mu_n}\mathbb{B} :n\geq j\right),$$ 
where $\alpha_n:= \frac{4\sqrt{\epsilon_n}}{\mu_n}+h(t,x_n(\delta_n(t)))+\sqrt{\gamma}$ and $\beta_n:= \frac{4\sqrt{\epsilon_n}}{\mu_n}+h(t,x_n(\delta_n(t)))$. On the other hand, there exists $(y_{n_j})$ which converges to $\dot{x}$ almost everywhere in $[0,T]$. Then, using Lemma \ref{upper-semicontinuity}, Lemma \ref{lemma_tech_1} and \ref{FixH1}, we have 
$$\dot{x}(t)\in -(h(x(t))+\sqrt{\gamma})\partial d_{C}(x(t)) +(h(x(t))+\sqrt{\gamma})\mathbb{B}\cap F(t,x(t))  \text{ for a.e. } t\in [0,T].
$$
Finally, since $\partial d_{C}(x(t))\subset N(C;x(t))$ for all $t\in [0,T]$, we obtain that  $x$ is the solution of \eqref{Fixed-sweep-2}.
\end{proof}

\section{Numerical methods for approximate projections}

As stated before, in most cases, finding an explicit formula for the projection onto a closed set is not possible. Therefore, one must resort to numerical algorithms to obtain approximate projections. There are several papers discussing this issue for different notions of approximate projections (see, e.g., \cite{pmlr-v139-usmanova21a}).  These algorithms are called \emph{projection oracles} and  provide an approximate solution $\bar{z}\in \H$ to the following optimization problem:
\begin{equation*}\tag{$P_x$}\label{opt-P}
\min_{z\in C} \|x-z\|^2,
\end{equation*}
where $C$ is a given closed set and $x\in\H$. Whether the approximate solution $\bar{z}$ belongs to the set $C$ or not depends on the notion of approximate projection. In our case, to implement our algorithm, we need that $\bar{z}\in C$.  In this line, a well-known projection oracle fulfilling this property can be obtained via the celebrated Frank-Wolfe algorithm (see, e.g., \cite{frank1956algorithm, pmlr-v28-jaggi13}), where a linear sub-problem of \eqref{opt-P} is solved in each iteration. For several types of convex sets, this method has been successfully developed (see \cite{pmlr-v28-jaggi13, MR4275646, MR4314104}). Besides, in \cite{ding2018frank}, it was shown that an approximate solution of the linear sub-problem is enough to obtain a projection oracle. 

Another important approach to obtaining approximate projections is the use of the Frank-Wolfe algorithm with separation oracles (see \cite{dadush2022simple}). Roughly speaking, a separation oracle determines whether a given point belongs to a set and, in the negative case, provides a hyperplane separating the point from the set (see \cite{MR0936633} for more details). For particular sets, it is easy to get an explicit separation oracle (see \cite[p. 49]{MR0936633}). An important example is the case of a sublevel set: let $g\colon \H\to\R$ be a continuous convex function and $\lambda\in\R$.  Then $[g\leq \lambda] := \{x\in\H : g(x)\leq\lambda \}$ has a separation oracle described as follows: to verify that any point belongs to $[g\leq \lambda]$ is straightforward. When a point $x\in\H$ does not belong to $[g\leq \lambda]$, we can consider any $x^\ast\in \partial g(x)$.  Then,  for all $y\in [g\leq\lambda]$,
\begin{equation*}
    \langle x^\ast,x\rangle\geq g(x)-g(y)+\langle x^\ast,y\rangle> \langle x^\ast,y\rangle,
\end{equation*}
where we have used that $g(x)>\lambda\geq g(y)$. Hence, the above inequality shows the existence of the desired hyperplane, which provides a separation oracle for $[g\leq \lambda]$. Therefore, if $C$ is the sublevel set of some convex function, we can use the algorithm proposed in \cite{dadush2022simple} to get an approximate solution $\bar{z}\in \proj_S^\epsilon(x)$.  Moreover, the sublevel set enables us to consider the case 
$$C(t,x) := \bigcap_{i=1}^m\{x\in \H:g_i(t,x)\leq 0\}=\{x\in\H:g(t,x):=\max_{i=1,...,m}g_i(t,x)\leq 0\},$$
where for all $t\in [0,T]$, $g_i(t,\cdot)\colon \H\to \R$, $i=1,\ldots,m$ are convex functions. We refer to 
 \cite[Proposition 5.1]{MR4027814} for the proper assumptions on these functions to ensure the Lipschitz property of the map $t\tto C(t)$ holds \eqref{lipschitz_hausdorff}.

\section{Concluding remarks}
In this paper, we have developed an enhanced version of the catching-up algorithm for sweeping processes through an appropriate concept of approximate projections. We provide the proposed algorithm's convergence for three frameworks: prox-regular, subsmooth, and merely closed sets. Some insights into numerical procedures to obtain approximate projections were given mainly in the convex case. Finally, the convergence of our algorithm for other notions of approximate solutions will be explored in forthcoming works.

%\bibliography{bib}   % name your BibTeX data base

\bibliographystyle{spmpsci}      % mathematics and physical sciences
\bibliography{bib}   % name your BibTeX data base


\end{document}